\documentclass[11pt]{amsart}
\usepackage{amsmath,amssymb,amsfonts,amsthm}
\usepackage{epsfig}
\usepackage{enumerate,color}
\usepackage{subfigure}
 \usepackage{amsaddr}
\newtheorem{theorem}{Theorem}[section]
\newtheorem{lemma}[theorem]{Lemma}

\newtheorem{claim}[theorem]{Claim}
\newtheorem{conjecture}[theorem]{Conjecture}
\newtheorem{observation}[theorem]{Observation}
\theoremstyle{definition}

\def\cL{\mathcal{L}}
\def\cC{\mathcal{C}}
\def\cF{\mathcal{F}}
\def\cR{\mathcal{R}}
\def\cS{\mathcal{S}}
\def\bbR{\mathbb{R}}

\let\oldrceil\rceil
\renewcommand{\rceil}{\right\oldrceil}
\let\oldlceil\lceil
\renewcommand{\lceil}{\left\oldlceil}

\begin{document}

\title[Coloring non-crossing strings]{Coloring non-crossing strings}

\author[L. Esperet]{Louis Esperet} 
\address{G-SCOP, CNRS \& Univ. Grenoble Alpes, Grenoble, France}
\email{louis.esperet@grenoble-inp.fr}
\thanks{A preliminary version of this work appeared in the proceedings of EuroComb'09~\cite{EGL09}.\\
  Louis Esperet is partially supported by ANR Project Stint
  (\textsc{anr-13-bs02-0007}) and LabEx PERSYVAL-Lab
  (\textsc{anr-11-labx-0025-01}). Arnaud Labourel is partially supported by ANR project MACARON (\textsc{anr-13-js02-0002}).}

\author[D. Gon\c calves]{Daniel Gon\c calves} 
\address{LIRMM (Universit\'e Montpellier 2, CNRS), Montpellier, France}
\email{goncalves@lirmm.fr}

\author[A. Labourel]{Arnaud Labourel} 
\address{LIF, Aix-Marseille Universit\'e \& CNRS, France}
\email{arnaud.labourel@lif.univ-mrs.fr}

\date{}

\dedicatory{}
\sloppy

\begin{abstract}
For a family $\cF$ of geometric objects in the plane, define $\chi(\cF)$ as the least integer
$\ell$ such that the elements of $\cF$ can be colored with $\ell$
colors, in such a way that any two intersecting objects have distinct
colors. When $\cF$ is a set of pseudo-disks that may only intersect on their boundaries, and
such that any point of the plane is contained in at most $k$
pseudo-disks, it can be proved that $\chi(\cF)\le 3k/2 + o(k)$ since
the problem is equivalent to cyclic coloring of plane graphs. In this
paper, we study the same problem when pseudo-disks are replaced by a
family $\cF$ of pseudo-segments (a.k.a. strings) that do not cross. In
other words, any two strings of $\cF$ are only allowed to ``touch''
each other. Such a family is said to be $k$-touching if no point of
the plane is contained in more than $k$ elements of $\cF$. We give
bounds on $\chi(\cF)$ as a function of $k$, and in particular we show
that $k$-touching segments can be colored with $k+5$ colors. This
partially answers a question of Hlin\v en\'y (1998) on the chromatic
number of contact systems of strings.
\end{abstract}

\maketitle

\section{Introduction}

For a family $\cF=\{S_1,\ldots,S_n\}$ of subsets of a set
$\Omega$, the \emph{intersection graph} $G(\cF)$ of $\cF$ is
defined as the graph with vertex-set $\cF$, in which two vertices
are adjacent if and only if the corresponding sets have non-empty
intersection.

For a graph $G$, the \emph{chromatic number} of $G$, denoted
$\chi(G)$, is the least number of colors needed in a proper coloring
of $G$ (a coloring such that any two adjacent vertices have distinct
colors). When talking about a proper coloring of a family $\cF$
of subsets of a given set, we implicitly refer to a proper coloring
of the intersection graph of $\cF$, thus the chromatic number
$\chi(\cF)$ is defined in a natural way. 

The chromatic number of families of geometric objects in the plane
have been extensively studied since the sixties
\cite{AG60,GL85,Kos88,KK97,McG00}. Since it is possible to construct
sets of pairwise intersecting (straight-line) segments of any size,
the chromatic number of sets of segments in the plane is unbounded in
general. However, Erd\H{o}s conjectured that triangle-free
intersection graphs of segments in the plane have bounded chromatic
number (see \cite{Gya87}). This was recently
disproved~\cite{PK12}. The conjecture of Erd\H{o}s initiated the study
of the chromatic number of families of geometric objects in the plane
as a function of their \emph{clique number}, the maximum size of
subsets of the family that pairwise intersect~\cite{FP08}. In this
paper, we consider families of geometric objects in the plane for
which the chromatic number only depends on local properties of the
families, such as the maximum number of objects containing a given
point of the plane.

\medskip

Consider a set $\cF=\{\cR_1, \ldots, \cR_n\}$ of pseudo-disks (subsets
of the plane which are homeomorphic to a closed disk) such that the
intersection of the interiors of any two pseudo-disks is empty. Let
$\mathcal{H}_\cF$ be the planar hypergraph with vertex set $\cF$, in
which the hyperedges are the maximal sets of pseudo-disks whose
intersection is non-empty. A proper coloring of $\cF$ is equivalent to
a coloring of $\mathcal{H}_\cF$ in which all the vertices of each
hyperedge have distinct colors. If every point is contained in at most
$k$ pseudo-disks, Borodin conjectured that there exists such a
coloring of $\mathcal{H}_\cF$ with at most $\frac32 k$
colors~\cite{Bor84}. It was recently proved that this conjecture holds
asymptotically~\cite{AEH13} (not only
in the plane, but also on any fixed surface). As a
consequence, $\cF$ can be properly colored with $\frac32 k+o(k)$
colors.

\begin{figure}[htbp]
\begin{center}
\hspace{0.05cm}
\subfigure[\label{fig:strings}]{\includegraphics[scale=0.7]{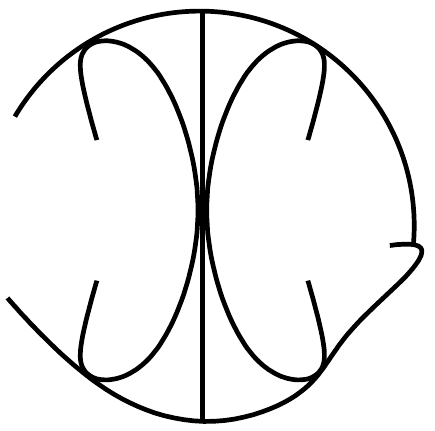}}
\hspace{2cm}
\subfigure[\label{fig:contact}]{\includegraphics[scale=0.75]{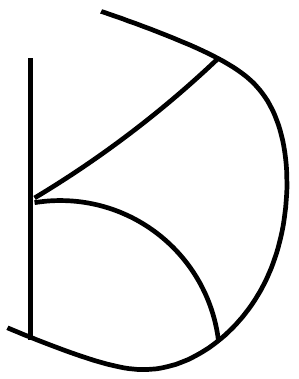}}
\caption{(a) A 3-touching set of strings $\cS_1$ with $G(\cS_1)\cong K_5$. (b) A one-sided 3-contact representation of
    curves $\cS_2$ with $G(\cS_2)\cong
    K_4$. \label{fig:regstrings}}
\end{center}
\end{figure}

It seems natural to investigate the same problem when pseudo-disks are
replaced by \emph{pseudo-segments}. These are continuous injective
functions from $[0,1]$ to $\bbR^2$ and are usually referred to as
\emph{strings}.  Consider a set $\cS=\{\cC_1, \ldots, \cC_n\}$ of such
strings.  We will always assume that any two strings intersect in a
finite number of points. We say that $\cS$ is \emph{touching} if no pair of strings
of $\cS$ cross, and that it is \emph{$k$-touching} if furthermore at
most $k$ strings can ``touch'' in any point of the plane, i.e., any
point of the plane is contained in at most $k$ strings (see
Figure~\ref{fig:strings} for an example).

Note that the family of all touching sets of strings contains all
\emph{contact systems of strings}, defined as sets of strings such
that the interior of any two strings have empty intersection. In other
words, if $c$ is a contact point in the interior of a string $s$, all
the strings containing $c$ distinct from $s$ end at
$c$. In~\cite{Hli98}, Hlin\v en\'y studied contact system of strings
such that all the strings ending at $c$ leave from the same side of
$s$. Such a representation is said to be \emph{one-sided} (see
Figure~\ref{fig:contact} for an example). It was proved
in~\cite{Hli98} that if a contact system of strings is $k$-touching and
every contact point is one-sided, then the strings can be colored with
$2k$ colors.

\medskip

In this paper, our aim is to study $k$-touching sets of strings in
their full generality. Observe that if $\cS$ is $k$-touching,
$k$ might be much smaller than the maximum degree of $G(\cS)$. 
However, based on the cases of pseudo-disks and contact
system of strings, we conjectured the following in the
conference version of this paper~\cite{EGL09}:

\begin{conjecture}\label{conj:multi}
For some constant $c>0$, any $k$-touching set of strings
can be colored with $ck$ colors.
\end{conjecture}

This conjecture was subsequently proved by Fox and Pach~\cite{FP10},
who showed that any $k$-touching set of strings can be colored with
$6ek+1$ colors (where $e$ is the base of the natural logarithm). In
Section~\ref{sec:general}, we show how to slightly improve their bound for small
values of $k$.  We also show that for any odd $k$, the clique on
$\tfrac92 (k-1)$ vertices can be represented as a set of $k$-touching strings,
so the best possible constant $c$ in Conjecture~\ref{conj:multi} is
between $4.5$ and $6e\approx 16.3$.

In Section~\ref{sec:intersect}, we give improved bounds when any two
strings can intersect a bounded number of times. In
Section~\ref{sec:contact}, we restrict ourselves to contact systems of
strings where any two strings intersect at most once (called
$1$-intersecting), which were previouly studied by Hlin\v
en\'y~\cite{Hli98}. He asked whether there is a constant $c$ such that
every one-sided $1$-intersecting $k$-touching contact system of strings is
$(k+c)$-colorable. We prove that they are $(\tfrac{4k}3
+6)$-colorable, and that every $k$-touching contact system of
straight-line segments is
$(k+5)$-colorable. Note that we do not need our contact
systems to be one-sided.

\medskip

Before giving general bounds, let us first mention two classical families
of touching strings for which coloring problems are well
understood. 

\smallskip

If a $k$-touching set of strings has the property that the interior
of each string is disjoint from all the other strings, then each
string can be thought of as an edge of some (planar) graph with maximum
degree $k$. By a classical theorem of Shannon, the strings can then be
colored with $3k/2$ colors. If moreover, any two strings intersect at
most once, then they can be colored with $k+1$ colors by a theorem of Vizing
 (even with $k$ colors whenever $k\ge 7$,
using a more recent result of Sanders and Zhao~\cite{SZ01}). In this
sense, all the problems considered in this article can be seen as a
extension of edge-coloring of planar graphs.

\smallskip

An \emph{$x$-monotone string} is a string such that every
vertical line intersects it in at most one point. Alternatively, it
can be defined as the curve of a continuous function from an interval
of $\bbR$ to $\bbR$. Sets of $k$-touching $x$-monotone strings are closely related to bar
$k$-visibility graphs. A \emph{bar $k$-visibility graph} is a graph
whose vertex-set consists of horizontal segments in the plane (bars),
and two vertices are adjacent if \emph{and only if} there is a
vertical segment connecting the two corresponding bars, and
intersecting no more than $k$ other bars. It is not difficult to see
that the graph of any set of $k$-touching $x$-monotone strings is a
spanning subgraph of some bar $(k-2)$-visibility graph, while any bar
$(k-2)$-visibility graph can be represented as a set of $k$-touching
$x$-monotone strings. Using this correspondence, it directly follows
from~\cite{DEG07} that $k$-touching $x$-monotone strings are $(6k-6)$-colorable, and that the complete graph on $4k-4$ vertices can be represented
as a set of $k$-touching $x$-monotone strings. If the left-most point
of each $x$-monotone string intersects the vertical line $x=0$, then
it directly follows from~\cite{FM08} that the strings can be colored
with $2k-1$ colors (and the complete graph on $2k-1$ vertices can be represented
by $k$-touching $x$-monotone strings in this specific way).

\section{General bounds}\label{sec:general}

Before proving our first results on the structure of sets of $k$-touching
strings in general, we make two important observations:

\begin{observation}\label{obs:pla}
The family of intersection graphs of 2-touching strings is exactly the
class of planar graphs.
\end{observation} 

The class of planar graphs being exactly the class of intersection
graphs of 2-touching pseudo-disks (see \cite{Koe36}) it is clear
that planar graphs are intersection graphs of 2-touching strings (by
taking a connected subset of the boundaries of each pseudo-disk). Furthermore, every intersection graph of 2-touching strings
is contained in an intersection graph of 2-touching pseudo-disks,
and is thus planar. Indeed, it is easy given a set of 2-touching
strings $\cS=\{\cC_1, \ldots, \cC_n\}$ to draw a set of 2-touching
pseudo-disks $\cF=\{\cR_1, \ldots, \cR_n\}$ such that $\cC_i \subset
\cR_i$ for every $i\in[1,n]$.

\begin{observation}\label{obs:dra}
We can assume without loss of generality that the strings in any set
of $k$-touching strings are polygonal lines (i.e. each string is the
union of finitely many straight-line segments) and that no endpoint 
of a string of $\cS$ coincides with an intersection between strings of $\cS$.
\end{observation} 

To see this, take a set $\cS$ of $k$-touching strings and consider the
following graph $G$: the vertices are the contact points and the
endpoints of the strings of $\cS$, and the edges connect two points if
they are consecutive in some string of $\cS$. The resulting graph is planar,
but might contain multiple edges. Subdivide each edge once, and
observe that the resulting graph $H$ is a simple planar (finite) graph, and
each string of $\cS$ is the union of some edges of $H$. Since $H$ is a
simple planar graph, by F\'ary's theorem~\cite{Far48} it has an
equivalent drawing in which all the edges are (straight-line) segments. 
If some endpoint of a string of $\cS$ coincides with an intersection
$x$ between strings of $\cS$, one can prolong each string ending at
$x$ with a small enough segment, and Observation~\ref{obs:dra} follows.

\medskip

The following result was proved by Fox and Pach~\cite{FP10} (this
shows Conjecture~\ref{conj:multi}). In the following, $e$ is the base of the natural logarithm.

\begin{theorem}[\cite{FP10}]\label{th:fox}
Any $k$-touching set of strings is $(6ek+1)$-colorable.
\end{theorem}

Their theorem is a direct consequence of the following bound on the
number of edges in a graph represented by a set of $k$-touching
strings. Their proof is inspired by the probabilistic proof of the
Crossing Lemma.

\begin{lemma}[\cite{FP10}]\label{lem:fox}
Any graph represented by a set of $n$ $k$-touching strings has less than $3ekn$ edges.
\end{lemma}

%% \begin{proof}
%% Let $\cS$ be a set of $n$ $k$-touching strings, and let $m$ be the
%% number of edges in the corresponding graph. For any two intersecting
%% strings $a$ and $b$, choose an arbitrary point $P(a,b)$ where $a$ and
%% $b$ intersect. We now select each string of $\cS$ uniformly at random,
%% with probability $p=\frac{1}{k}$. Let $\cS'$ be the set of selected
%% strings. For each pair of intersecting strings $a,b$ of $\cS$, the
%% probability that $a$ and $b$ were both selected and that none of the
%% other strings containing $P(a,b)$ were selected is at least
%% $p^2(1-p)^{k-2}$. It follows that the expectation of the number of such pairs is at
%% least $p^2(1-p)^{k-2}m$. Slightly modify each intersection point $x$ of
%% $\cS'$ contained in more than two strings by making all these strings
%% disjoint around $x$. The graph corresponding to this new set of
%% strings still has at least $p^2(1-p)^{k-2}m$ edges on average, and is
%% expected to contain $pn$ vertices. Since it is 2-touching, it is
%% planar and thus $p^2(1-p)^{k-2}m< 3pn$. Using that
%% $(1-\frac1k)^{k-2}\ge e^{-1}$ for any $k\ge 3$, we obtain that $m<3ekn$.
%% \end{proof}

%% \begin{proofof}\textbf{Theorem~\ref{th:fox}.}
%% If the intersection graph has less than $3ekn$ edges, its average degree
%% is less than $6e k$. Furthermore, since the class of graphs defined
%% by $k$-touching strings is closed under
%% taking induced subgraphs, the intersection graph is $6ek$-degenerate and
%% thus $(6ek+1)$-colorable.
%% \end{proofof}

When $k=3$, their proof can easily be optimized to show that the
number of edges is less than $12n$. Hence, every such graph has a
vertex with degree less than 24. These graphs are thus 23-degenerate
and have chromatic number at most 24. We now show how to modify their
proof to slightly improve this bound.

\begin{theorem}\label{th:k3}
Any 3-touching set of strings is 19-colorable.
\end{theorem}

This is a direct consequence of the following lemma.

\begin{lemma}\label{lem:k3}
Any graph represented by a set of $n$ 3-touching strings has at most
$\tfrac67(6+\sqrt{22})n\approx 9.16 \,n$ edges.
\end{lemma}

\begin{proof}
Let $\cS$ be a set of $n$ 3-touching strings, and let $m$ be the
number of edges in the corresponding intersection graph $G$. By
Observation~\ref{obs:dra}, we can assume that each string of $\cS$ is
a polygonal line (a union of finitely many segments) and
 that no endpoint of a string of $\cS$ coincides with an intersection between strings of $\cS$.
Consider a point $p$ contained in three different
strings $s_0,s_1,s_2$. Then there is a small disk $D$ centered in $p$ that only intersects the strings $s_0,s_1,s_2$, and is such that the
boundary $C$ of $D$ intersects each $s_i$ (for $i=0,1,2$) in exactly two points,
say $p_i,p_i'$. Assume that walking around $C$ in clockwise order, we
see $p_0,p_0',p_1,p_1',p_2,p_2'$. Then for $i=0,1,2$, we replace $s_i\cap D$
by a new string $s'_i$ between $p_i$ and $p_i'$ as
follows. For each $i=0,1,2$, let $q_i$ be a point of $D$ that is after
$p_i'$ and before $p_{i+1}$ in clockwise order (with indices modulo
3). For two points $x,y$, let $S(x,y)$ denote the (straight-line) segment
between $x$ and $y$. Then we replace $s_i\cap D$ (the portion of $s_i$ between $p_i'$
and $p_i$) by the concatenation of $S(p_i',q_i)$, $S(q_i,q_{i-1})$,
and $S(q_{i-1},p_i)$ (see Figure~\ref{fig:sandwich}). Note that the resulting set of strings is still
3-touching, and the intersection graph remains unchanged. Repeating this
operation if necessary, we can assume without loss of generality that for any three strings
$s_0,s_1,s_2$ as above, we see $p_0,p_0',p_1',p_2',p_2,p_1$ when
walking around $C$ in clockwise order. In this case we say that $s_1$
is ``sandwiched'' between $s_0$ and $s_2$ at $p$ (see Figure~\ref{fig:transform}, left). 

\begin{figure}[htbp]
\begin{center}
\includegraphics[scale=1.4]{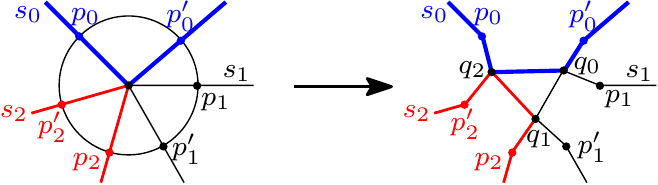}
\caption{A local modification turning one 3-touching point into
  three 2-touching points,
  without changing the intersection graph. \label{fig:sandwich}}
\end{center}
\end{figure}

We now define two spanning subgraphs $G_1$ and $G_0$ of $G$ as follows.
Two strings $a$ and $c$ are adjacent in $G_1$ if they intersect and for every
intersection point $p$ of $a$ and $c$, there exists a third string $b$ that is 
sandwiched between $a$ and $c$ at $p$. For every pair of strings $a$
and $c$ adjacent in $G_1$, let $P_1(a,c)$ be an arbitrarily chosen
intersection point between $a$ and $c$. By definition, there is a
unique string that is sandwiched between $a$ and $c$ at $P_1(a,c)$. 

Two strings $a$ and $c$ are adjacent in $G_0$ if they are adjacent in $G$ but not in $G_1$,
i.e., if there exists an intersection point $P_0(a,c)$ of $a$ and $c$
such that either $P_0(a,c)$ is not contained in another string or $P_0(a,c)$ is  contained in a third string $b$ that is not sandwiched between $a$ and $c$. For $i=0,1$, the edge-set of $G_i$ is denoted by $E_i$, and the cardinality of $E_i$ is denoted by $m_i$.
%
%
%For any intersection point $p$, we now define a notion of distance (relative to $p$) on the
%strings intersecting $p$: if the strings are $a,b,c$ in this order
%(with $b$ sandwiched between $a$ and $c$), we say that $a$ and $b$
%(resp. $b$ and $c$) are at distance 0, while $a$ and $c$ are at
%distance 1. If $p$ is only contained in two strings, then their
%distance (relative to $p$) is also said to be 0.
%
%For any two intersecting strings $a$ and $b$ their (absolute) distance
%is the minimum, among their intersection points $p$, of their distance
%relative to $p$. Let $P(a,b)$ be one of the points where this minimum
%is realized. The set of
%edges $uv$ where $u,v$ are at distance $i$ ($i=0,1$) is denoted by $E_i$,
%and the cardinality of $E_i$ is denoted by $m_i$.

\begin{claim}\label{cl1}
$7m_0\ge m+6n$
\end{claim}

For each edge $ab \in E_1$, assume that $ab$ gives a charge of 1 to
the unique string $c$ sandwiched between $a$ and $b$ at $P_1(a,b)$. Let
$N_0(c)$ be the set of neighbors of $c$ in $G_0$. The total charge $\rho(c)$ received by $c$ is at most the number
of pairs of vertices $x,y \in N_0(c)$ such that in $\cS\setminus
\{c\}$, $P_0(x,y)$ is a 2-contact point. If we modify the set of strings
intersecting $c$ so as to only preserve those 2-touching points (all
the other pairs of strings are made disjoint), we obtain a planar graph
with vertex-set $N_0(c)$ and with at least $\rho(c)$ edges. It follows that
$\rho(c)\le 3|N_0(c)|-6$. Summing for all strings $c$, we obtain that
the total charge $\sum_{c \in \cS} \rho(c)=m_1\le 6 m_0 - 6n$. Since
$m_0+m_1=m$, we have $7m_0\ge m+6n$, as claimed.

\medskip

\begin{figure}[htbp]
\begin{center}
\includegraphics[scale=1]{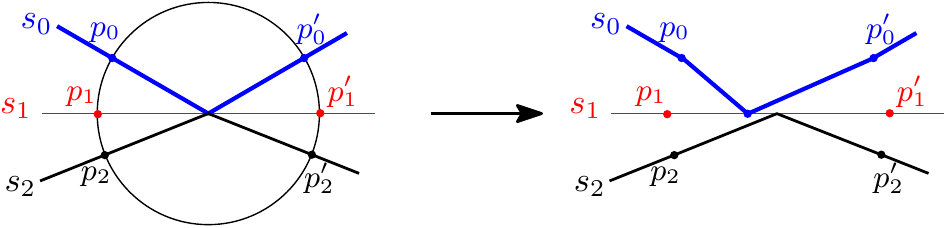}
\caption{A local modification preserving adjacency in the graph $G_0$. \label{fig:transform}}
\end{center}
\end{figure}

We are now ready to prove the lemma. We select each string of $\cS$
uniformly at random with probability $p$ (to be chosen later). In the
subset of chosen strings, we slightly modify each 3-touching point $x$
as
follows. There is a small disk $D$ with boundary $C$
centered in $x$ that only intersects the three strings $s_0,s_1,s_2$ containing
$x$, and (with the same notation as before), we can assume without loss of generality that we see $p_0,p_0',p_1',p_2',p_2,p_1$ when
walking around $C$ in clockwise order. Then for some point $q$ on
$s_1$, close to $x$, we replace $s_0\cap D$ by the concatenation
of the segments $S(p_0,q)$ and $S(q,p_0')$ (see Figure~\ref{fig:transform}). Note that this preserves all edges
in $G_0$. Let $\cS'$ be the new set of strings.  For each edge
$ab\in E_0$, the probability that $ab$ appears as an edge in $\cS'$ is
$p^2$, and for each edge $ab\in E_1$, the probability that $ab$
appears as an edge in $\cS'$ is at least the probability that $a$ and
$b$ were both selected and the string $c$ sandwiched between $a$ and
$b$ at $P_1(a,b)$ was not selected, which is $p^2(1-p)$. The expected
number of vertices of the graph represented by $\cS'$ is then $pn$ and
its expected number of edges is at least
%% $$m_0 p^2+m_1p^2(1-p)\ge
%% \tfrac{p^2}7(m+6n+(1-p)(6m-6n))=\tfrac{p^2}7(m(7-6p)+6pn).$$ 
%%
%%  ! We don't have $7m_1\ge 6m-6n$ !

\begin{eqnarray*}
m_0 p^2+m_1p^2(1-p) & = & p^2 \left(m_0 + m_1 (1-p)\right)\\
& = & p^2 \left(m_0 p + m (1-p)\right)\\
& \ge & \tfrac{p^2}7\left((m+6n) p + 7m (1-p) \right)\\
& \ge & \tfrac{p^2}7(m(7-6p)+6pn)\\
\end{eqnarray*}

By
Observation~\ref{obs:pla}, the graph is planar and therefore,
$\tfrac{p^2}7(m(7-6p)+6pn)\le 3pn$. This can be rewritten as $m\le
n\,\frac{21-6p^2}{p(7-6p)}$. Taking $p=3-\sqrt{11/2}$, we obtain
$m\le \tfrac67(6+\sqrt{22})n$.
\end{proof}

The ideas of Lemma~\ref{lem:k3} can be used to slightly improve the
multiplicative constant in Theorem~\ref{th:fox} for all $k$. Since our
improvement is minor (we obtain a bound of $6k\times 2.686$ instead of
$6k\times 2.718$), we omit the details.

\begin{figure}[htbp]
\begin{center}
\hspace{0.05cm}
\subfigure[\label{fig:exstring1}]{\includegraphics[scale=0.6]{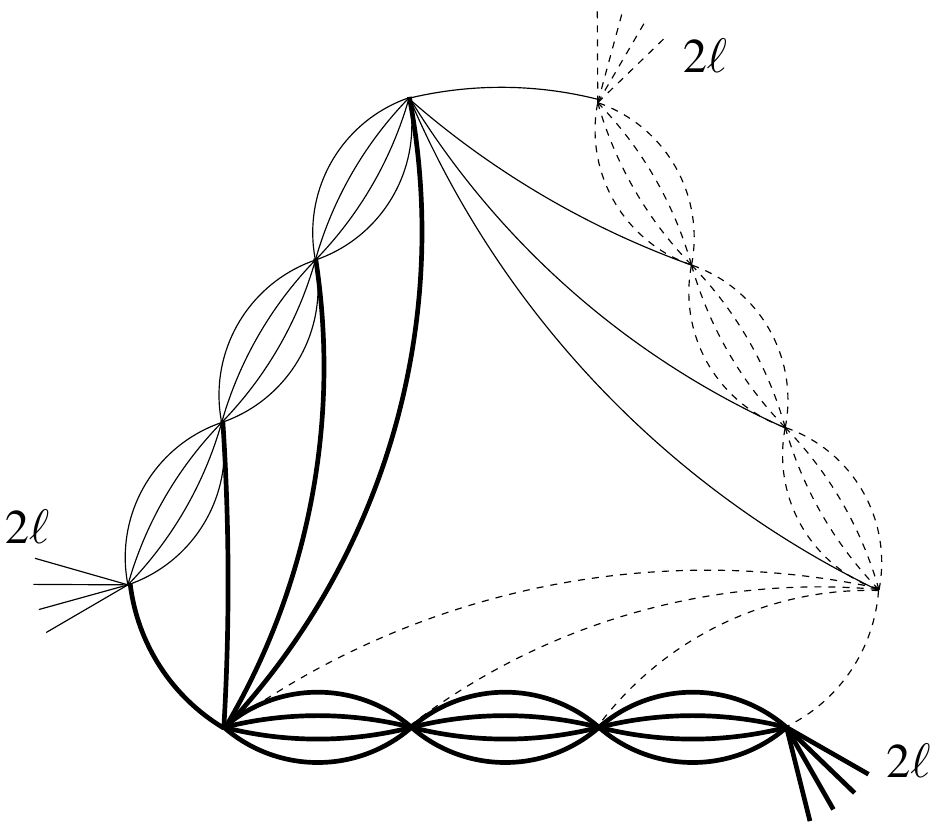}}
\hspace{1cm}
\subfigure[\label{fig:exstring2}]{\includegraphics[scale=0.6]{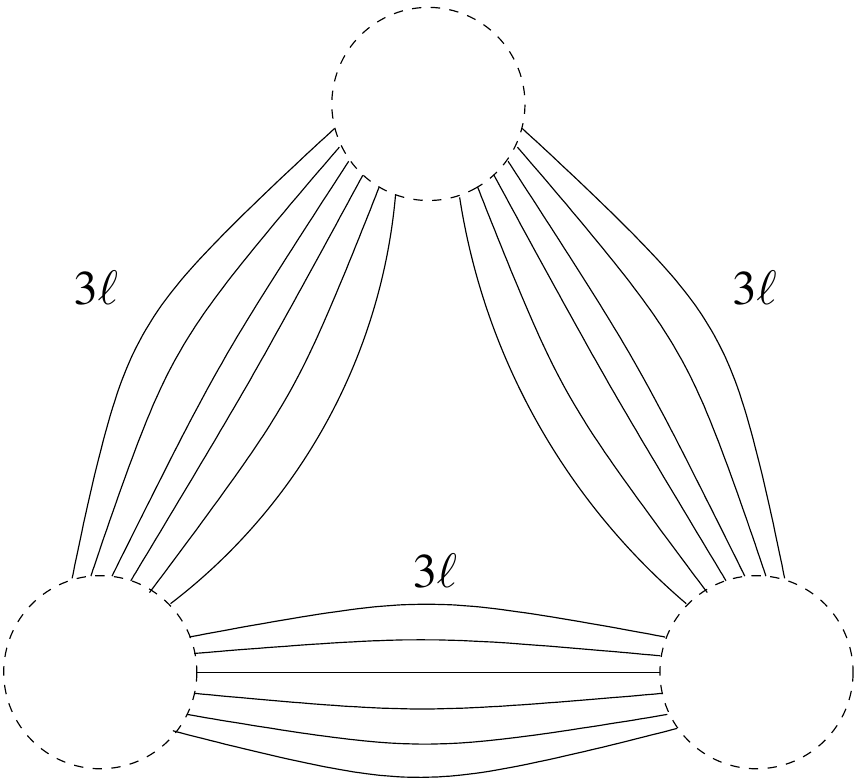}}
\caption{(a) The construction of a $2\ell$-sun. (b) A set $\cS$ of
  $k$-touching strings requiring $\lceil \frac{9 k}2 \rceil -5$
  distinct colors (each dashed circle represents a $2\ell$-sun). \label{fig:exstring}}
\end{center}
\end{figure}

We now show that the constant $c$ in Conjecture~\ref{conj:multi} is
at least $\frac92$. 

\begin{theorem}
For every odd $k\ge 1$, $k=2\ell+1$, there exists a set of $k$-touching
strings $\cS_k$ such that the strings of $\cS_k$ pairwise touch and
such that $|\cS_k|= 9\ell = 9(k-1)/2$. Thus $\chi(\cS_k) = 9\ell = \lceil
\frac{9 k}2 \rceil -5$.
\end{theorem}

\begin{proof} Consider $n$ touching strings $s_1,\ldots,s_n$ that all intersect
$n$ points $c_1,\ldots, c_n$ in the same order (see the set of bold
strings in Figure \ref{fig:exstring1} for an example when $n=4$), and
call this set of strings an \emph{$n$-braid}. For some $\ell>0$, take
three $2 \ell$-braids $\cS_1$, $\cS_2$, $\cS_3$, and for
$i=1,2,3$, connect each of the strings of $\cS_i$ to a different
intersection point of $\cS_{i+1}$ (with indices taken modulo 3),
while keeping the set of strings touching (see Figure
\ref{fig:exstring1}). We call this set of touching strings a
\emph{$2\ell$-sun}. Observe that a $2\ell$-sun contains $6\ell$ strings that pairwise intersect, and that each intersection point
contains at most $2\ell+1$ strings. Moreover, each of the $6\ell$
strings has an end that is incident to the infinite face.

We now consider three $2\ell$-suns $\cR_1,\cR_2,\cR_3$. Each of them
has $6\ell$ strings with an end incident to the outerface. For each $i=1,2,3$, we arbitrarily divide
the strings leaving $\cR_i$ into two sets of $3\ell$ consecutive strings, say $R_{i,i+1}$
and $R_{i,i-1}$. For each $i=1,2,3$, we now take the strings of
$R_{i,i+1}$ and $R_{i+1,i}$ by pairs (one string in $R_{i,i+1}$, one string
in $R_{i+1,i}$), and identify these two strings into a single
string. This can be made in such a way that the resulting set of
$(6\times 3\ell)/2=9\ell$ strings is still $(2\ell+1)$-touching (see Figure \ref{fig:exstring2}, where the three
$2\ell$-suns are represented by dashed circles, and only the portion
of the strings leaving the suns is displayed for the sake of clarity). Hence we obtain a
$k$-touching set of $\lceil \frac{9 k}2 \rceil -5$ strings that
pairwise intersect, as desired.
\end{proof}

\section{$\mu$-intersecting strings}\label{sec:intersect}

Let $\cS$ be a $k$-touching set of strings. The set $\cS$ is said to
be \emph{$\mu$-intersecting} if any two strings intersect in at most
$\mu$ points. We denote by $H(\cS)$ the multigraph associated to
$\cS$: the vertices of $H(\cS)$ are the strings of $\cS$, and two
strings with $t$ common points correspond to two vertices connected by
$t$ edges in $H(\cS)$. Note that the intersection graph $G(\cS)$ of
$\cS$ is
the simple graph underlying $H(\cS)$.

We prove the following result (which only supersedes
Theorem~\ref{th:fox} for $\mu \le 5$).

\begin{theorem}\label{th:chro}
Any $k$-touching set $\cS$ of $\mu$-intersecting strings can be
properly colored with $3 \mu k$ colors.
\end{theorem}

Again, the proof is based on an upper bound on the number of edges of such
graphs.

\begin{lemma}\label{lem:sbsup}
If $\cS$ is a $k$-touching set of $n$ $\mu$-intersecting strings,
then $H(\cS)$ (and so, $G(\cS)$) has less than $\frac{3}2 \mu k n$ edges.
\end{lemma}

\begin{proof}
Let $n$ denote the number of strings of $\cS$, and let $N$ denote the
number of intersection points of $\cS$. Let us denote
by $d(c)$ the number of strings containing an intersection point $c$
(for any $c$, $2 \le d(c) \le k$ by definition). 

%% Recall that by
%% Observation~\ref{obs:sandwich}, we can assume that all intersection
%% points are sandwich points (this assumption is not necessary, but it helps
%% making things clearer).
%%
%% En fait non !

\medskip

 By
Observation~\ref{obs:dra}, we can assume that each string of $\cS$ is
a polygonal line (a union of finitely many segments) and that no 
endpoint of a string of $\cS$ coincides with an intersection between 
strings of $\cS$.
Let us slightly modify $\cS$ in order to obtain a set $\cS'$ of
2-touching and $\mu$-intersecting strings. For that, repeat the
following operation while there exists an intersection point $c$ with
$d(c)>2$. There is a small disk $D$ centered in $c$ such that
$D$ only intersects the strings containing $c$, and for any such string
$s$, $s\cap D$ is the union of two segments. Pick a string $s_1$
containing $c$ such that the angle between its two
segments at $c$ is minimal, and let $s_2$ be a string containing
$c$, distinct from $s_1$, such that a face $f$ of $D-\cS$ is bounded by $s_1$, $s_2$ and the
boundary of $C$. Let $p_1,p_1'$ be the intersection of $s_1$ and the
boundary of $D$, and let $q$ be a point of $s_2\cap f$ distinct from
$c$ (and close to $c$). Then we replace $D\cap s_1$ by the concatenation of the
straight-line segments $S(p_1,q)$ and $S(q,p_1')$ (see Figure~\ref{fig:muinter}). This is similar to
the second modification used in Theorem~\ref{th:k3} and illustrated in
Figure~\ref{fig:transform}, which was restricted to the case where $c$
is contained in exactly three strings.

\begin{figure}[htbp]
\begin{center}
\includegraphics[scale=1]{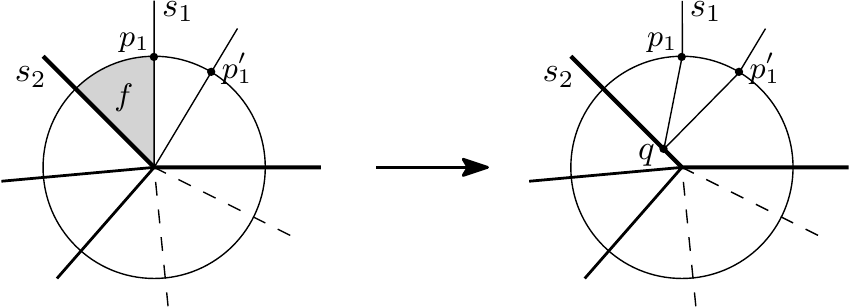}
\caption{A local modification reducing the number of strings
  containing a given point. \label{fig:muinter}}
\end{center}
\end{figure}

Each intersection point $c$ in $\cS$ corresponds to a set $X_c$ of
intersection points in $\cS'$. Let $N'$ be the number of intersection
points in $\cS'$. Since $\cS'$ is 2-touching, $N'$ is also the number
of edges of the multigraph $H(\cS')$. By construction, each $X_c$ has size exactly
$d(c)-1$, hence $N'=\sum_c |X_c|= \sum_c (d(c)-1)$.  By
Observation~\ref{obs:pla}, the graph $G(\cS')$ and the multigraph $H(\cS')$ are
planar, and since $\cS'$ is $\mu$-intersecting, $H(\cS')$
is a planar multigraph in which each edge has multiplicity at most
$\mu$, therefore it contains $N' \leq
(3n-6)\mu$ edges.  As $d(c)\le k$ for any intersection point $c$ in $\cS$,
we have $$\sum_c d(c)(d(c)-1) \, \le \, kN' \, \le \, (3n-6)\mu k \, <
\,3 \mu k n .$$

Finally, since the number of edges of $H(\cS)$ is precisely $\frac12 \sum_c
d(c)(d(c)-1)$, we have that $H(\cS)$ has less than $\frac32 \mu k n$
edges, as desired.
\end{proof}

In particular, if a $k$-touching set $\cS$ of strings is such
that any two strings intersect in at most one point,
Theorem~\ref{th:chro} yields a bound of $3 k$ for the chromatic
number of $\cS$. We suspect that it is far from tight:

\begin{conjecture}\label{conj:uni}
There is a constant $c>0$, such that every $k$-touching set of
1-intersecting strings can be colored with $k+c$ colors.
\end{conjecture}

In the next section, we show that this conjecture holds for
$k$-touching (straight-line) segments, a special case of
1-intersecting strings. It is interesting to note that even though
the bound for $k$-touching $\mu$-intersecting graphs in
Conjecture~\ref{conj:multi} and Theorem~\ref{th:fox} does not depend
on $\mu$, the chromatic number of these graphs has some connection
with $\mu$: sets of strings with $\mu=1$ have chromatic number at most
$3 k$, whereas there exists sets of strings with large $\mu$ and
chromatic number at least $\frac92 (k-1)$.

\begin{figure}[htbp]
\begin{center}
\hspace{0.05cm}
\subfigure[\label{fig:sbinf1}]{\includegraphics[scale=0.6]{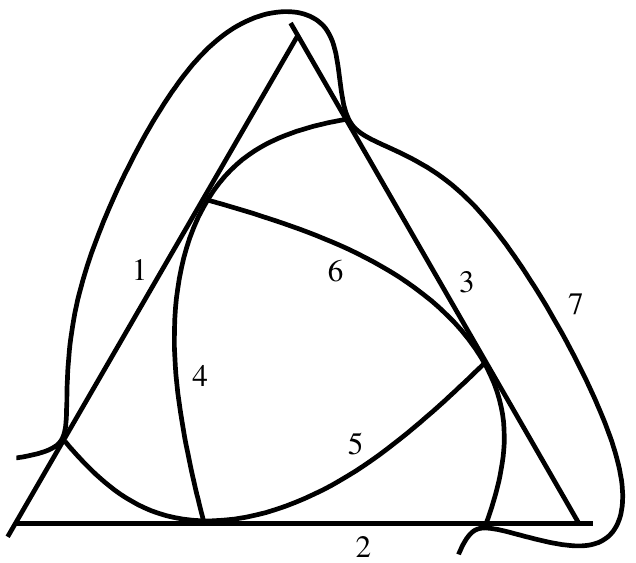}}
\hspace{2cm}
\subfigure[\label{fig:sbinf2}]{\includegraphics[scale=0.6]{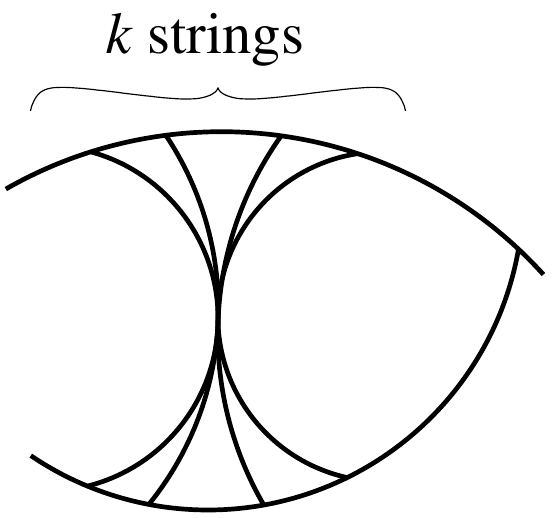}}
\caption{(a) A 3-touching set of 1-intersecting strings requiring 7 colors
  (b) A $k$-touching set of 1-intersecting strings requiring $k+2$
  colors (here $k=4$). \label{fig:sbinf}}
\end{center}
\end{figure}

Note that the constant $c$ in Conjecture~\ref{conj:uni} is at least
4. Figure \ref{fig:sbinf1} depicts a 3-touching set of seven
1-intersecting strings, in which any two strings intersect. Hence,
this set requires seven colors. However this construction does not
extend to $k$-touching sets with $k\ge 4$, it might be that the
constant is smaller for higher $k$.  In Figure \ref{fig:sbinf2}, the
$k$-touching set $\cS_{k}$ contains $k+2$ non-crossing strings, and is
such that any two strings intersect. Hence, $k+2$ colors are required
in any proper coloring.

\section{1-Intersecting contact system of strings}\label{sec:contact}

In this section, all the sets of strings we consider are
1-intersecting (any two strings intersect in at most one point).
An interesting example of 1-intersecting set of strings
is any family of non-crossing (straight-line) segments in the
plane. Such a family is also known as a \emph{contact system of
  segments}, and has been studied in~\cite{FOP91}, where the authors
proved that any bipartite planar graph has a contact representation
with horizontal and vertical segments. 

More generally, a \emph{contact system of strings} is a family of
strings such that the interiors of the strings are pairwise
non-intersecting. In other words, if $c$ is a contact point in the
interior of a string $s$, all the strings containing $c$ and distinct from
$s$ end at $c$. A contact point $p$ is a {\it peak} if every string
containing $p$ has an end at $p$. Otherwise, that is if $p$ is an
interior point of a string $s$ and an end for all the other strings
containing $p$, $p$ is {\it flat}. A flat contact point $p$ is {\it
  one-sided} if all the strings ending at $p$ are on the same side of
the unique string whose interior contains $p$. A contact system of
strings in which every flat contact point is one-sided is also said to
be one-sided.

It was proved by Hlin\v en\'y~\cite{Hli98a} that the intersection graph
of any one-sided 2- or 3-touching set of segments is planar. Note that
as a 2-touching set of segments is always one-sided, it is also
4-colorable.  In~\cite{Oss99}, Ossona de Mendez proved that
it is NP-complete to determine whether a 2-touching set of segments is
3-colorable.

In~\cite{Hli98}, Hlin\v en\'y studied the clique and chromatic numbers
of one-sided $k$-touching contact systems of strings. He proved that
the maximal clique in this class is $K_{k+1}$ and that the graphs in
this class are $2k$-colorable. He also asked the following: is there a
constant $c$ such that if a contact system of strings is $k$-touching,
1-intersecting, and one-sided, then it is $(k+c)$-colorable? Note that
Conjecture~\ref{conj:uni} would imply a positive answer to this
question.

\medskip

In the first part of this section, we prove that 1-intersecting and
$k$-touching contact systems of strings are $\left(\lceil\tfrac43
k\rceil+6\right)$-colorable. In the second part, we show that any
$k$-touching contact system of segments is $(k+5)$-colorable. Note
that we do not assume the contact systems to be one-sided (but we also
show that adding this assumption slightly improves the additive
constants in our results).

\begin{theorem}\label{th:contact}
For any $k\ge 3$, any 1-intersecting $k$-touching contact system of strings can be
colored with $\lceil \tfrac43 k\rceil+6$ colors.
\end{theorem}

As in Theorem~\ref{th:fox}, the result is a consequence of a bound on
the degeneracy of these graphs:

\begin{lemma}\label{lem:contact}
For any $k\ge 3$, if $\cS$ is a 1-intersecting $k$-touching contact system of
strings, then $G(\cS)$ contains a vertex of degree at most $\lceil \tfrac43 k\rceil+5$.
\end{lemma}
\begin{proof}
Assume that there is a counterexample, i.e. a 1-intersecting
$k$-touching contact system $\cS$
of $n$ strings such that $G(\cS)$ has minimum degree at least $\lceil
\tfrac43 k\rceil+6$. In particular, $G(\cS)$ has $m \ge n(\tfrac23
k+3)$ edges. We take a counterexample for which $n$ is minimal, and
with respect to this, $m$ is maximal.
Observe that $G(\cS)$ is connected, since otherwise by minimality
of $n$, some connected component would have a vertex of degree at most $\lceil \tfrac43 k\rceil+5$, a contradiction.
Observe also that if some string of $\cS$ has at most one contact
point, then the corresponding vertex of $G(\cS)$ has degree at most
$k-1\le \lceil \tfrac43 k\rceil+5$, a contradiction. This implies that every string of $\cS$ has at least two contact
points. As a consequence, we can also assume that the two
ends of each string of $\cS$ are contact points (if not, delete the
portion of a string between a free end and its first contact point).

\smallskip

Let $H({\cS})$ be the plane graph whose vertices are the contact
points of $\cS$, whose edges link two contact points if and only if they are
consecutive on a string of $\cS$, and whose faces are the connected
regions of $\bbR^2\setminus \cS$.

Let $p_i$ and $f_i$ be the number of contact points of exactly $i$
strings of $\cS$ that are respectively peaks and flat. Let us denote by $c$ the total number of contact points, and note that $c=\sum_{i=2}^{k}(p_i+f_i)$.
By counting the number of ends of a string of $\cS$ in two different
ways, we obtain that:

\begin{equation}
\label{eq1}
2n = \sum_{i=2}^{k} i p_i + \sum_{i=2}^{k} (i-1)f_i
\end{equation}

Consider a one-sided flat contact point $p$ and let $s$ be the unique
string such that $p$ is an interior point of $s$. If we draw a small
open disk $D$ containing $p$, a unique face $f$ of $H(\cS)$ has the
property that $f \cap D$ is incident to $s$, and to no other string
containing $p$. We denote this face $f$ of $H(\cS)$ by $F(p)$. Remark
that since $\cS$ is 1-intersecting, any face $f$ of $H(\cS)$ contains at least
$|F^{-1}(f)|+3$ vertices, thus at least $|F^{-1}(f)|$ edges can be
added to $H(\cS)$ (inside $f$) with the property that $H(\cS)$ remains
planar. Hence in total, one can add as many edges to $H(\cS)$ as the number of
one-sided contact points, while keeping $H(\cS)$ planar. Since every flat
2-contact point is one-sided, and every planar graph on $c$ vertices has at
most $3c-6$ edges, it follows that $H(\cS)$ has at most $3c-6-f_2$
edges. As a consequence, the sum of the degrees of the vertices of
$H(\cS)$ is:

$$\sum_{i=2}^{k} i p_i + \sum_{i=2}^{k} (i+1)f_i \leq 2\cdot (3c - 6 - f_2).$$

By the definition of $c$, this is equivalent to:

\begin{equation}
\label{eq2}
\sum_{i=2}^{k} (i-6) p_i + \sum_{i=2}^{k} (i-5)f_i \leq  - 12 -2 f_2
\end{equation}
Since any pair of strings in $\cS$ intersects at most once, the number
of edges in $G(\cS)$ satisfies the following equation.
\begin{equation}
\label{eq3}
m = \sum_{i=2}^{k} {i \choose 2}(p_i + f_i)
\end{equation}

Our goal is to use (\ref{eq1}), (\ref{eq2}), and (\ref{eq3}) to
prove that $m$ is less than $n(\tfrac23
k+3)$, which is a contradiction. To prove this, we will see that in
order to maximize $m$, we have to set all the values of $f_i$ and
$p_i$ to zero, except for $p_2,f_3,p_k,f_k$ (in other words, the
weight has to be concentrated on the extremal variables). Once this is
proved, bounding $m$
will be significantly simpler.

Let us consider the linear program (LP1) defined on variables $p_i$
and $f_i$ with values in $\mathbb{R}^+$ such that the
equation~(\ref{eq1}) and the inequality (\ref{eq2}) are satisfied, and
such that the value $m$ defined by (\ref{eq3}) is maximized. Here $n$
is considered to be a constant (it is not a variable of the linear program). Note that the
solution $m^*$ of this problem is clearly an upper bound of the number
of edges of $G(\cS)$.

\begin{claim}\label{claim8}
The optimal solutions of (LP1) are such that $f_2 = 0$.
\end{claim}
If $f_2 \neq 0$, take a small $\epsilon>0$ and replace $f_2$ by
$f_2-\epsilon$ and $f_3$ by $f_3+\epsilon/2$. Then (\ref{eq1}) remains
valid, inequality (\ref{eq2})
still holds (both sides are increased by $2\epsilon$), while (\ref{eq3}) is increased by $\epsilon/2$. This concludes the proof of the claim.

\begin{claim}\label{claim9}
The optimal solutions of (LP1) are such that $f_i = 0$,
for every $4\leq i\leq k-1$.
\end{claim}
If for some $4\leq i\leq k-1$, $f_i\ne 0$, choose a small $\epsilon>0$
and replace (i) $f_{3}$ by $f_{3}+\epsilon\,\frac{(k-i)}{k-3}$; (ii)
$f_{i}$ by $f_{i}-\epsilon$; and (iii) $f_{k}$ by
$f_{k}+\epsilon\,\frac{i-3}{k-3}$. Then (\ref{eq1}) remains valid,
inequality (\ref{eq2}) still holds (the left-hand side and the
right-hand side remain unchanged), while the value of $m$ in
(\ref{eq3}) is increased by $\frac{\epsilon}{k-3}(3(k-i)-{i
  \choose 2}(k-3)+{k \choose 2}(i-3))$. The function $g: i\mapsto 3(k-i)-{i
  \choose 2}(k-3)+{k \choose 2}(i-3)$ is a (concave) parabola with
$g(4)={k \choose 2}-3k+6>0$ (recall that $4\leq i\leq k-1$, so $k\ge 5$) and $g(k)=0$, so it is positive in the interval $[4,k-1]$. This concludes the proof of the claim.

\begin{claim}\label{claim7}
The optimal solutions of (LP1) are such that $p_i = 0$, for every
$3\leq i \leq k-1$.
\end{claim}
If for some $3\leq i\leq k-1$, $p_i\ne 0$, choose a small $\epsilon>0$
and replace (i) $p_{2}$ by $p_2+\epsilon\,\frac{(k-i)}{k-2}$; (ii)
$p_{i}$ by $p_{i}-\epsilon$; and (iii) $p_{k}$ by
$p_{k}+\epsilon\,\frac{i-2}{k-2}$. Then (\ref{eq1}) remains valid,
inequality (\ref{eq2}) still holds (the left-hand side and the
right-hand side remain unchanged), while the value of $m$ in
(\ref{eq3}) is increased by $\frac{\epsilon}{k-2}((k-i)-{i
  \choose 2}(k-2)+{k \choose 2}(i-2))$. The function $g: i\mapsto (k-i)-{i
  \choose 2}(k-2)+{k \choose 2}(i-2)$ is a (concave) parabola with
$g(3)={k \choose 2}-2k+3>0$ (recall that $3\leq i\leq k-1$, so $k\ge
4$) and $g(k)=0$, so it is positive in the interval $[3,k-1]$. This concludes the proof of the claim.

\bigskip

It follows from the previous claims that $c=p_2+f_3+p_k+f_k$, so
(\ref{eq2}) gives $-4p_2+(k-6)p_k-2f_3+(k-5)f_k<0$. Therefore,
$$(k-6)(p_k+f_k)<4p_2+2f_3.$$ By equation (\ref{eq1}) and the previous claims, we have $2n=2p_2+k p_k
+2f_3+(k-1)f_k$. Hence, $$2p_2+f_3 \le
2p_2+2f_3=2n-k p_k-(k-1)f_k \le 2n-(k-1)(p_k+f_k),$$ which implies
$(k-6)(p_k+f_k)<4n-2(k-1)(p_k+f_k)$. This can be rewritten
as $$(3k-8)(p_k+f_k)<4n.$$ Now, by equation (\ref{eq3}),
$$
\begin{array}{rcl}
m=p_2+3f_3+{k \choose 2}(p_k+f_k) & \le & \frac32(2n-(k-1)(p_k+f_k))+{k
    \choose 2}(p_k+f_k) \\
  & \le & 3n+(p_k+f_k)(-\frac32(k-1)+{k \choose 2})\\
  & < &  3n+\frac{4n}{3k-8}(3k-8)\frac{k}6 \mbox{\hspace{10pt} (since }k\ge 3\mbox{)}\\
  & < & n\,(\tfrac23 k+3).
\end{array}$$

We obtain that the graph $G(\cS)$ has less than $n(\tfrac23 k+3)$
edges, which is a contradiction.
\end{proof}

If the contact system we consider is one-sided, the argument we used
for flat 2-intersection points while establishing (\ref{eq2}) in the
previous proof works for all flat points, and it follows that $H(\cS)$
has at most $3c-6-\sum_{i=2}^k f_i$ edges. Consequently, inequality (\ref{eq2})
can be replaced by the following stronger inequality:
\begin{equation}
\label{eq4}
\sum_{i=2}^{k} (i-6) p_i + \sum_{i=2}^{k} (i-3)f_i \leq - 12
\end{equation}
Let (LP2) be the new linear program. It is not difficult to check that
Claims~\ref{claim8} and~\ref{claim9} remain satisfied. Moreover, it
can be proved that in some optimal
solution of (LP2), we also have $f_3=0$. Similar computations give
that in this case the graph $G(\cS)$ contains a vertex of degree at most $\lceil \tfrac43 k\rceil+1$. As a consequence:

\begin{theorem}\label{th:contactos}
For any $k\ge 3$, any 1-intersecting one-sided $k$-touching contact system of strings can be
colored with $\lceil \tfrac43 k\rceil+2$ colors.
\end{theorem}

We now show how to modify the proof of Theorem~\ref{th:contact} to
prove that $k$-touching contact systems of segments are
$(k+5)$-colorable. For technical reasons, we will consider instead the
case of \emph{extendible} contact system of strings,
which we define next.

\medskip

A \emph{pseudo-line} is the homeomorphic image of a straight
line in the plane. An \emph{arrangement of pseudo-lines} is a set of
pseudo-lines such that any two of them intersect at most once, and when they do they cross each other. We say that a contact system of strings
is \emph{extendible} if there is an arrangement of pseudo-lines, such
that each string $s$ of the contact system is contained in a distinct
pseudo-line of the arrangement (this pseudo-line is called \emph{the
  support} of $s$). Observe that a contact system of
segments is clearly extendible. On the other hand, it was proved by de
Fraysseix and Ossona de Mendez ~\cite{dFOdM07} that any extendible
contact system of strings can be ``stretched'', i.e. continuously
changed to a contact system of segments, while keeping the same
underlying intersection graph.

We start with two observations about extendible contact systems of
strings (the first one is similar to the first part of Observation~\ref{obs:dra}).

\begin{observation}\label{obs:drapl}
Let $\cS$ be an extendible contact system of
strings, and let $\cL$ be a corresponding arrangement of
pseudo-lines. We can assume without loss of generality that each
string of $\cS$ is the union of finitely many segments,
and each pseudo-line of $\cL$ is the union of two rays (affine images
of $\{0\}\times[0,+\infty)$ in $\mathbb{R}^2$) and finitely many segments.
\end{observation} 

The proof is similar to that of Observation~\ref{obs:dra}.  We consider the plane graph $G$ whose vertices are the endpoints of
the strings of $\cS$ and the intersection points of $\cS \cup \cL$, and whose edges connect two points if
they are consecutive in some string of $\cS$ or some pseudo-line of
$\cL$. Since any two pseudo-lines intersect at most once, the resulting graph $H$ is a simple planar (finite) graph, and
by F\'ary's theorem~\cite{Far48} it has an
equivalent drawing in which all the edges are (straight-line)
segments. We can then replace each end of a pseudo-line by finitely
many segments and a ray,
without changing $G(\cS)$, and
Observation~\ref{obs:drapl} follows.

\begin{observation}\label{obs:drapl2}
Let $\cS$ be an extendible contact system of
strings, and let $\cL$ be a corresponding arrangement of
pseudo-lines. We can assume without loss of generality that for any
strings $s_1,s_2,s_3\in \cS$, if their supporting pseudo-lines
intersect in a point $p$, then $s_1,s_2,s_3$ also intersect in $p$.
\end{observation} 

Consider a point $p$ contained in $\ell_1,\ell_2,\ell_3 \in \cL$, and
such that the string $s_1\in \cS$ supported by $\ell_1$ does not contain
$p$. We now show how to modify $\ell_1$ so that it avoids $p$, without
changing the other elements of $\cL$ and the elements of $\cS$. By Observation~\ref{obs:drapl}, we can assume that
a small disk $D$ centered in $p$ only intersects the elements of $\cS
\cup \cL$ containing $p$, and the intersection of each of these
elements with $D$ consists of the union of at most two segments. By the definition of a
pseudo-line arrangement, the pseudo-lines intersecting $p$ pairwise
cross. Note
that $\ell_1$ cuts the boundary of $D$ in two arcs, say $a,a'$. We then
modify $\ell_1$ by replacing $D\cap \ell_1$ by $a$. The set $\cS$ is not
modified, and the pseudo-lines
intersecting $p$ still pairwise cross, so we obtain a new pseudo-line
arrangement. Repeating this operation if necessary, we finally obtain
a pseudo-line arrangement satisfying Observation~\ref{obs:drapl2}.

\medskip

We can now prove the main result of this section.

\begin{theorem}\label{th:seg}
For any $k\ge 3$, any $k$-touching extendible contact system of strings can be properly
colored with $k+5$ colors.
\end{theorem}

As before, the result is a consequence of a bound on the degeneracy of
the corresponding intersection graphs:

\begin{lemma}\label{lem:seg}
For any $k\ge 3$, if $\cS$ is a $k$-touching extendible contact system
of strings, then $G(\cS)$ contains a vertex of degree at most $k+4$.
\end{lemma}

\begin{proof}
The proof is similar to that of Theorem~\ref{th:contact}. We
consider a counterexample $\cS$ consisting
of $n$ $k$-touching extendible strings. Since $G(\cS)$ has minimum
degree at least $k+5$, $G(\cS)$ has $m
\ge \tfrac12 (k+5)\,n$ edges. Again, we take
$n$ minimal, and with respect to this, $m$ maximal. As before, we
can assume that $G(\cS)$ is connected and that the two ends of
each string are contact points.

\smallskip

We again consider the plane graph $H({\cS})$ whose vertices are the contact
points of $\cS$, whose edges link two contact points if and only if they are
consecutive on a string of $\cS$, and whose faces are the connected
regions of $\bbR^2\setminus \cS$. Let $p_i$ and $f_i$ be the number of contact points of exactly $i$
strings of $\cS$ that are respectively peaks and flat. Let
$p=\sum_{i=2}^{k}p_i$ and $c=\sum_{i=2}^{k}(p_i+f_i)$.

Recall that any face $f$ of $H(\cS)$ contains at least
$|F^{-1}(f)|+3$ vertices (where $F^{-1}(f)$, defined in the proof of
Lemma~\ref{lem:contact}, is the number of flat contact points
``incident'' to $f$), thus at least $|F^{-1}(f)|$ edges can be
added to $H(\cS)$ (inside $f$) with the property that $H(\cS)$ remains
planar. We now show that, moreover, the vertices
corresponding to the peaks of $\cS$ all lie on the outerface of
$H({\cS})$. This directly implies that we can add $f_2+p-3$ edges to
$H(\cS)$, while keeping $H(\cS)$ planar (as a consequence, $H({\cS})$ has at most
$3c-6-(f_2+p-3)=3c-3-f_2-p$ edges).

Indeed, if some peak $x$ of $\cS$ is not incident to the outerface, we
choose a string $s$ containing $x$ and prolong $s$ after $x$
(following the pseudo-line supporting $s$) until
it hits some other segment $s'$ (note that since $\cS$ is extendible,
the pseudo-lines supporting $s$ and $s'$ intersect at most once, and
thus $s$ and $s'$ did not intersect previously). Let $\cS'$ be the new
contact system of strings. It is still extendible, and by
Observation~\ref{obs:drapl2} we can assume that $s\cap s'$ is only
contained in $s$ and
$s'$, so the new contact point is 2-touching. Consequently, this new contact
system is extendible and $k$-touching, which contradicts the
maximality of $m$.

It follows that inequality (\ref{eq2}) in the proof of
Theorem~\ref{th:contact} can be replaced by:
\begin{equation}
\label{eq5}
\sum_{i=2}^{k} (i-4) p_i + \sum_{i=2}^{k} (i-5)f_i \leq - 6-2 f_2
\end{equation}

We consider the linear program (LP3) defined on variables $p_i$
and $f_i$ with values in $\mathbb{R}^+$ such that the
equation~(\ref{eq1}) and the inequality (\ref{eq5}) are satisfied, and
such that the value $m$ defined by (\ref{eq3}) is maximized.

The coefficients of the variables $f_i$ being the same in (\ref{eq2})
and (\ref{eq5}), Claim~\ref{claim8} (the optimal solutions of (LP3) are such that $f_2
= 0$) and Claim~\ref{claim9} (the optimal solutions of (LP3) are such that $f_i = 0$,
for every $4\leq i\leq k-1$) are still satisfied. Claim~\ref{claim7} is
now replaced by the following stronger claim:

\begin{claim}\label{claim11}
The optimal solutions of (LP3) are such that $p_i = 0$, for every
$2\leq i \leq k$.
\end{claim}
If for some $2\leq i\leq k$, $p_i\ne 0$, choose a small $\epsilon>0$
and replace $p_{i}$ by $p_i-\epsilon$, and $f_{i}$ by
$f_{i}+\epsilon\,\frac{i}{i-1}$. Then (\ref{eq1}) remains valid,
inequality (\ref{eq5}) still holds (the left-hand side is decreased by
$\frac{4\epsilon}{i-1}$ and the right-hand side remains unchanged), while the value of $m$ in (\ref{eq3})
is increased by $\frac{\epsilon i}{2}$. This concludes the proof of the claim.

\bigskip

It follows from Claims~\ref{claim8}, \ref{claim9}, and~\ref{claim11} that $c=f_3+f_k$, so
(\ref{eq5}) gives $(k-5)f_k<2f_3$. By equation (\ref{eq1}), $2f_3=2n-(k-1)f_k$, which implies $(k-5)f_k<2n-(k-1)f_k$. This
can be rewritten as $$(k-3)f_k<n.$$ Now, by equation
(\ref{eq3}),
$$
\begin{array}{rcl}
m=3f_3+{k \choose 2}f_k & \le & \frac32(2n-(k-1)f_k)+{k
    \choose 2}f_k \\
  & \le & 3n+f_k\tfrac{(k-1)(k-3)}{2}\\
  & < & n(\tfrac{k-1}2 +3).
\end{array}
$$
We obtain that the graph $G(\cS)$ has less than $\tfrac{n}2(k+5)$
edges, which is a contradiction.
\end{proof}

It follows that intersection graphs of $k$-touching segments are
$(k+4)$-degenerate and then $(k+5)$-colorable. Note that the
$(k+4)$-degeneracy may not be tight for every $k$. Indeed, we only
know graphs that are not $(k+3)$-degenerate for $k\le 6$. Those graphs
are obtained in the following way for $k=6$: First consider the
segments in a straight-line drawing of a planar triangulation with
minimum degree 5 and maximum degree 6, such that any degree five
vertex is at distance at least two from the outerface, and such that
any two vertices of degree five are at distance at least three
apart. Such a graph can be obtained from the icosahedron by
applying several times the following operation: subdivide every edge once, and inside
each face, add 3 edges connecting the 3 (newly created) vertices of
degree 2. It follows from Euler's formula that there are precisely 12
vertices of degree 5 in the triangulation, and therefore precisely 60 segments whose ends
respectively touch a 5- and a 6-contact points. Let
$s_1,\ldots,s_{60}$ be these segments. Those are the only
segments that touch less than 10 other segments (they only touch 9 of
them). To make each of them touch one more segment, prolong successively the
segments $s_1,s_2,\ldots,s_{60}$ by their end that is at the
6-contact point, until reaching another segment (we can assume that
the drawing of the triangulation is such that no line contains two
edges incident to the same vertex, therefore each of the segments
$s_1,\ldots,s_{60}$ hits an interior point of another segment).

\begin{figure}[htbp]
\begin{center}
\subfigure[\label{fig:segd2}]{\includegraphics[scale=0.4]{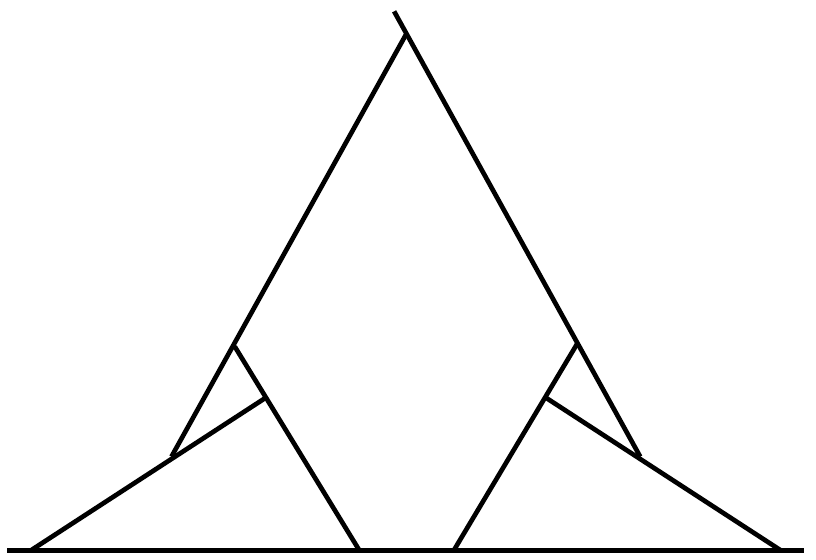}}
\hspace{1cm}
\subfigure[\label{fig:segd3}]{\includegraphics[scale=0.4]{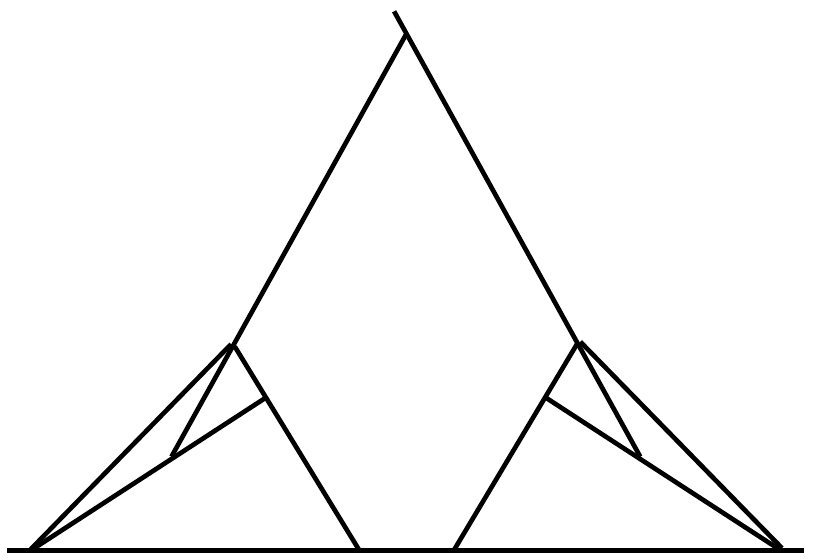}}
\hspace{1cm}
\subfigure[\label{fig:segdk}]{\includegraphics[scale=0.4]{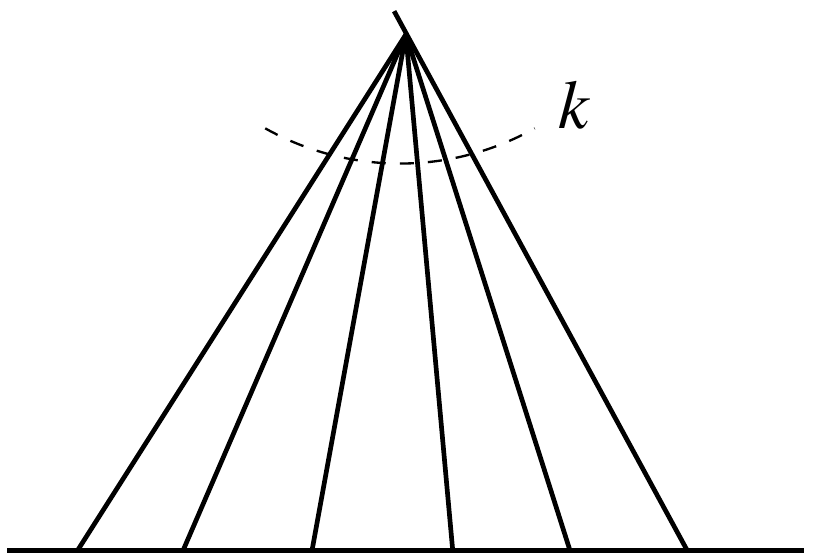}}
\caption{(a) A 2-touching set of segments requiring 4 colors (b) A
  3-touching set of segments requiring 5 colors (c) A $k$-touching set
  of segments requiring $k+1$ colors. \label{fig:segd}}
\end{center}
\end{figure}

Figure~\ref{fig:segd} depicts $k$-touching sets of segments requiring
$k+2$ colors, for $k=2,3$. However it does not appear to be trivial to
extend this construction for any $k \ge 4$. Note that
Figure~\ref{fig:segd3} also shows that there are intersection graphs
of 3-contact representations of segments (with two-sided contact
points) that are not planar (this remark also appears in~\cite{Hli95}).

\section*{Acknowledgments}

The authors would like to thank the reviewers for their suggestions
and remarks (in particular for finding an error in a previous version
of our manuscript).

\end{document}